\documentclass[11pt,a4paper,reqno]{article}
\usepackage[T1]{fontenc}
\usepackage[utf8]{inputenc}
\usepackage[english]{babel}
\usepackage{lmodern,microtype,amsmath,amssymb,amsthm,mathtools}
\usepackage[margin=27mm,headheight=14pt]{geometry}
\usepackage{enumitem,xcolor,fancyhdr}
\usepackage[colorlinks=true,linkcolor=blue!40!black,citecolor=blue!40!black,urlcolor=blue!40!black]{hyperref}
\usepackage{bookmark}
\setlist{itemsep=4pt,topsep=5pt}

\newtheorem{theorem}{Theorem}[section]
\newtheorem{lemma}[theorem]{Lemma}
\newtheorem{proposition}[theorem]{Proposition}
\newtheorem{corollary}[theorem]{Corollary}
\theoremstyle{definition}
\newtheorem{definition}[theorem]{Definition}

\theoremstyle{remark}
\newtheorem{remark}[theorem]{Remark}
\numberwithin{equation}{section}
\newcommand{\B}{\mathcal B(X)}
\newcommand{\F}{\mathcal F(X)}
\newcommand{\Q}{\mathcal Q(X)}
\newcommand{\Nil}{\mathcal N(X)}
\newcommand{\C}{\mathbb C}
\newcommand{\PP}{\mathbb P}
\newcommand{\Ione}{\mathcal I_1(X)}
\newcommand{\Rone}{\mathcal R_1(X)}
\DeclareMathOperator{\tr}{tr}
\DeclareMathOperator{\rank}{rank}
\DeclareMathOperator{\spanop}{span}
\DeclareMathOperator{\AlgRef}{Ref}
\newcommand{\spn}[1]{\spanop\{#1\}}
\newcommand{\step}[1]
{\par\smallskip\noindent\textit{#1}\quad}
\title{Surjective maps preserving quasi-nilpotence\\of  products of operators}
\author{Zakariae Aznay \qquad Fadil Chabbabi\\[1mm]
\small Laboratory (A.F.N.L.A), Department of Mathematics, Faculty of Sciences,\\
\small Abdelmalek Essa\^adi University, Tetouan, Morocco\\[1mm]
\small \texttt{aznayzakariae@gmail.com}\qquad \texttt{f.chabbabi@uae.ac.ma}}
\date{}
\begin{document}
\maketitle
\begin{abstract}
Let $X$ be an infinite-dimensional complex Banach space and let $\mathcal B(X)$ the full algebra of all bounded linear operator on $X$. In this paper, we  give a complete characterization of  surjective maps $\phi:\mathcal B(X)\to\mathcal B(X)$ satisfying
\[
r(AB)=0\quad\Longleftrightarrow\quad r(\phi(A)\phi(B))=0
\qquad(A,B\in\mathcal B(X)).
\]
Without assuming additivity, injectivity, or continuity, we show that there exists a function $\gamma : B(X)\mapsto \C \setminus\{0\}$ such that the map $\phi$  takes one of the following forms: 
\begin{enumerate}
    \item there exists a bounded linear or conjugate-linear, invertible operator  $S:X\to X$ such that :
\begin{equation*}\label{eq:direct}
\phi(A)=\gamma(A)SAS^{-1},\qquad \text{ for all } \quad A\in B(X),
\end{equation*}
\item the space $X$ is reflexive and there exists a bounded linear or conjugate-linear, invertible operator   $S:X^*\to X$ such that
\begin{equation*}\label{eq:dual}
\phi(A)=\gamma(A)SA^*S^{-1},\qquad \text{ for all } \quad A\in B(X).
\end{equation*}
\end{enumerate}
\end{abstract}

\section{Introduction and main result}
Throughout this paper, $X$ is an infinite-dimensional complex Banach space, $\B$ is the algebra of all bounded linear operators on $X$, 
and $X^*$ is  the dual of $X$, that is, the vector space of all continuous linear functionals on $X$. For $A\in\B$, its Banach adjoint $A^*:X^*\to X^*$ is defined by $A^*f=f\circ A$ for all $f\in X^*$. 
  
An operator $A$ is said to be \emph{quasi-nilpotent} when its spectral radius $r(A)$ is zero. We denote the sets of all quasi-nilpotent and nilpotent operators by $\Q$ and $\Nil$, respectively.

In \cite[Theorem 2.1]{LSS}, Li, \v{S}emrl, and Sze characterized surjective maps \(\phi:{\B}\to{\B}\) preserving nilpotence of products in both directions:
\begin{equation}\label{eq:SS}
AB\in \Nil\quad\Longleftrightarrow\quad \phi(A)\phi(B)\in \Nil
\qquad (A,B\in\mathcal{B}(X)).
\tag{1}
\end{equation}
In the infinite-dimensional case, they proved that  there exist a function
\[
\gamma:\B\longrightarrow\C\setminus\{0\}
\]
for which one of the following two cases holds:
\begin{enumerate}[label=\textup{(\alph*)}]
\item There exists a  bounded and invertible, linear or conjugate-linear $S:X\to X$  such that 
\begin{equation}\label{eq1}
\phi(A)=\gamma(A)SAS^{-1} \qquad \text{ for all } \quad A\in B(X).
\end{equation}
\item The space $X$ is reflexive, and there exists  a bounded and invertible, linear or conjugate-linear $S:X^*\to X$  such that 
\begin{equation}\label{eq2}
\phi(A)=\gamma(A)SA^*S^{-1}\qquad \text{ for all } \quad A\in B(X).
\end{equation}
\end{enumerate}
In the finite-dimensional case, when $ dim(X)\ge 3$ they showed that $\phi$ has an analogous form, up to a complex homomorphism.

In this paper, we study surjective maps $\phi:\B \to \B$ that preserve the quasi-nilpotence of the products of operators  in the following sense: 
\begin{equation}\label{eq:P}
AB\in \Q\quad\Longleftrightarrow\quad
\phi(A)\phi(B)\in \Q
\qquad(A,B\in\B).
\tag{2}
\end{equation}
We do not assume that \(\phi\) is additive or continuous. In Theorem~\ref{thm:main}, we prove that every such map has the same form as above. Consequently, the problems of characterizing surjective maps preserving nilpotency and quasinilpotency of operator products, as stated in \eqref{eq:SS} and \eqref{eq:P}, respectively, are equivalent.

The main results are stated in the following theorem.

\begin{theorem}\label{thm:main}
Let $\phi:\B\to\B$ be a surjective  map. The following statements are equivalent.
\begin{enumerate}[label=\textup{(\roman*)}]
\item The map $\phi$  satisfies  the condition \eqref{eq:P}.
\item There exist a function
\[
\gamma:\B\longrightarrow\C\setminus\{0\}
\]
for which one of the following two cases holds:
\begin{enumerate}[label=\textup{(\alph*)}]
\item there exists a  bounded and invertible, linear or conjugate-linear $S:X\to X$  such that 
\begin{equation}\label{eq:direct}
\phi(A)=\gamma(A)SAS^{-1} \qquad \text{ for all } \quad A\in B(X).
\end{equation}
\item The space $X$ is reflexive, and there exists  a bounded and invertible, linear or conjugate-linear $S:X^*\to X$  such that 
\begin{equation}\label{eq:dual}
\phi(A)=\gamma(A)SA^*S^{-1}\qquad \text{ for all } \quad A\in B(X).
\end{equation}
\end{enumerate}
Moreover, the scalar function satisfies
\begin{equation}\label{eq:radial}
\{\gamma(tA)\sigma(t):t\in\C^\times\}=\C^\times
\qquad(A\in\B\setminus\{0\}), 
\end{equation}
where the mapping $t\to \sigma(t)\in \C$ is either the identity or the complex conjugation.  
\end{enumerate}
\end{theorem}

In finite dimensions case, nilpotence and quasi-nilpotence are the same. Hence condition (1) is exactly the condition of preserving nilpotence of products of operators studied by Šemrl and coauthors.

The paper is organized as follows. Section~\ref{sec:known} recalls some basic properties of rank-one operators that will be used throughout the paper. Section~\ref{sec:lemmas} presents several useful results concerning the trace functional on the space of finite-rank operators. Section~\ref{sec:reflexive} recalls the notion of the algebraic reflexive hull and reviews a result characterizing the two-dimensional reduced linear subspaces of \(\mathcal{B}(U,V)\) that are not algebraically reflexive, where \(U\) and \(V\) are finite-dimensional vector spaces and \(\mathcal{B}(U,V)\) denotes the vector space of linear operator denotes the vector space of linear operators from $U$ to  $V$. This result is then used to establish a linear dependence criterion for three operators that are pointwise contained in the same subspace. Section~\ref{sec:incidence} introduces the incidence condition for subsets of projective spaces, which will be used to characterize rank-one operators. Finally, Section~\ref{sec:projective} contains the proof of Theorem~\ref{thm:main}, based on the preceding results.
\section{Notation and Preliminary Results on Rank-One Operators}\label{sec:known}
In the sequel, we write $\F$ for the vector space of finite-rank operators, $\Rone$ for the set of operators of rank one, and $\Ione$ for the set of rank-one idempotents. For $x\in X$ and $f\in X^*$, the operator $x\otimes f$ is given by
\[
(x\otimes f)y=f(y)x\qquad(y\in X).
\]
It has rank one when $x$ and $f$ are nonzero, moreover every rank one operator has this form. 

For an operator $R$ of rank at most one,  we have 
\begin{equation}\label{eq:rankone}
R^2=(\tr R)R,\qquad r(R)=|\tr R|,
\end{equation}
where $\tr R$ and $r(R)$  denote the standard trace and the spectral radius of $R$, respectively. In particular,
\begin{equation}\label{eq:rankonetests}
r(A(x\otimes f))=|f(Ax)|,
\qquad
r((x\otimes f)A)=|f(Ax)|.
\end{equation}
Consequently, for a rank-at-most-one operator, being quasinilpotent, being nilpotent, and having trace zero are equivalent.

For $R_1,R_2\in\mathcal R_1$, we write $R_1\sim R_2$ if they have
the same range or the same kernel. In particular, for rank-one
operators $x\otimes f$ and $y\otimes g$, we have
\(
x\otimes f\sim y\otimes g
\)
if and only if $x$ and $y$ are linearly dependent or $f$ and $g$
are linearly dependent. Moreover, if $R_1\sim R_2$, then $\lambda R_1\sim\beta R_2$
for all nonzero scalars $\lambda,\beta\in\mathbb C$.

The following three lemmas are taken from Šemrl's paper 
\cite[Lemmas 2.11--2.13]{LSS}. The first one provides a
necessary and sufficient condition for two rank-one operators
to be equivalent. The second characterizes the orthogonality
of projections in terms of this equivalence relation. The third lemma establishes a necessary and sufficient
condition for two nonzero operators to be linearly dependent.

\begin{lemma}\cite{LSS}\label{known:adjacency}
Let $R_1,R_2$ be linearly independent rank-one operators. Then $R_1\sim R_2$ if and only if there is a rank-one operator $R_3$, such that $R_3$ is linearly independent of $R_1$, $R_3$ is linearly
independent of $R_2$, and  for every $T\in\B$,
\[
R_1T,R_2T\in\Nil\quad\Longrightarrow\quad R_3T\in\Nil.
\]
\end{lemma}
\begin{lemma}\cite{LSS}\label{known:orthogonality}
Let $P,Q\in\Ione$ be distinct. Then $PQ=QP=0$ if and only if there are rank-one nilpotents $M,N$ such that
\[
P\sim M,\quad P\sim N,\quad Q\sim M,\quad Q\sim N,
\qquad M\not\sim N.
\]
\end{lemma}
\begin{lemma}\cite{LSS}\label{known:determination}
Let $A,B\in\B\setminus\{0\}$. If
\[
AP\in\Nil\quad\Longleftrightarrow\quad BP\in\Nil
\qquad(P\in\Ione),
\]
then $A$ and $B$ are linearly dependent.
\end{lemma}
\begin{remark}\label{rem:nil-quasi-nil}
Since every tested product in Lemmas~\ref{known:adjacency} and~\ref{known:determination} has rank at most one, the statements remain identical when the nilpotence conditions are written as quasi-nilpotence conditions. This observation may be used automatically below.
\end{remark}
Finally, we quote \cite[Theorem 2.4]{Semrl}, which provides a characterization of the bijective maps \(\xi:\Ione\to\Ione\) preserving orthogonality in both directions.
\begin{theorem}\cite{Semrl}\label{known:semrl}
Let $X$ be an infinite-dimensional complex Banach space. If $\xi:\Ione\to\Ione$ is a bijection satisfying
\[
PQ=QP=0\quad\Longleftrightarrow\quad
\xi(P)\xi(Q)=\xi(Q)\xi(P)=0,
\]
then either
\[
\xi(P)=SPS^{-1}\qquad(P\in\Ione)
\]
for a bounded invertible linear or conjugate-linear $S:X\to X$, or
\[
\xi(P)=SP'S^{-1}\qquad(P\in\Ione)
\]
for a bounded invertible linear or conjugate-linear $S:X^*\to X$. In the second case $X$ is reflexive.
\end{theorem}
All four rank-one results just quoted are used under their stated hypotheses. In particular, we will establish bijectivity and preservation in both directions before applying Theorem~\ref{known:semrl}.

\section{Some Results on the Trace Functional}
\label{sec:lemmas}

The trace functional plays a crucial role in the proofs of our main
results. In this section, we establish several auxiliary results
concerning the trace that will be useful in the sequel. 

\begin{lemma}\label{lem:tracepairing}
For $A\in\B$, define $\ell_A:\F\to\C$ by $\ell_A(T)=\tr(AT)$. The map $A\mapsto\ell_A$ is linear and injective. If $A_1,\ldots,A_m$ are linearly independent, then the maps
\[
\Psi : T\longmapsto(\ell_{A_1}(T),\ldots,\ell_{A_m}(T))
\]
 $\F$ is onto $\C^m$.
\end{lemma}
\begin{proof}
The linearity of $\ell_A$ follows directly from the linearity of the trace functional. 
To prove injectivity, suppose that $\ell_A=0$. Then, for every $x\in X$ and 
$f\in X^*$, we have
\[
0=\ell_A(x\otimes f)
=\operatorname{tr}(Ax\otimes f)
=f(Ax).
\]
Since $f(Ax)=0$ for every $f\in X^*$, the Hahn--Banach theorem implies that 
$Ax=0$ for every $x\in X$. Consequently, $A=0$, which establishes the 
injectivity of the map $A\mapsto \ell_A$.

If $\sum_j c_j\ell_{A_j}=0$, injectivity gives $\sum_j c_jA_j=0$, and the  independence of the $A_j$ then gives $c_j=0$ for every $j$. Thus the $m$ functionals are independent. Suppose, to the
contrary, that $\operatorname{Im}\Psi$ is a proper subspace of
$\mathbb C^m$. Since $\mathbb C^m$ is finite-dimensional, there
exists a nonzero linear functional
$\varphi:\mathbb C^m\to\mathbb C$ that vanishes on
$\operatorname{Im}\Psi$. Hence there exist scalars
$c_1,\ldots,c_m$, not all zero, such that
\(
\varphi(z_1,\ldots,z_m)=\sum_{j=1}^m c_jz_j
\)
and 
$
0=\sum_{j=1}^m c_j\ell_{A_j}(T),
$ for every $T\in\F.$
Consequently,
\(
\sum_{j=1}^m c_j\ell_{A_j}=0,
\)
which contradicts the linear independence of
$\ell_{A_1},\ldots,\ell_{A_m}$. Therefore,
$\operatorname{Im}\Psi=\mathbb C^m$, and $\Psi$ is surjective.
\end{proof}

\begin{lemma}\label{lem:plane}
Let $\mathcal V\subseteq\B$ be a two-dimensional linear subspace containing an operator of rank at least two. Then, there is $T\in\F$ such that
\[
r(CT)>0\qquad(C\in\mathcal V\setminus\{0\}).
\]
\end{lemma}
\begin{proof}
Choose a basis $A,B$ of $\mathcal V$ with $\rank A\ge2$. Note that neither $A$ nor $B$ is assumed to have finite rank. We divide the proof into several steps :

\step{Step 1: There is an operator \(T_0\in \F\) satisfying
\[
\tr(AT_0)=0
\quad\text{and}\quad
\tr\bigl((AT_0)^2\bigr)=2.
\]}

Since $\operatorname{rank}A\ge 2$, there exist $x,y\in X$ such that $Ax$ and $Ay$ are linearly independent. By the Hahn-Banach theorem, we may therefore choose $f,g\in X^*$ satisfying
\[
f(Ax)=g(Ay)=1,\qquad f(Ay)=g(Ax)=0.
\]
Set $T_0=x\otimes f-y\otimes g$. This is bounded and has finite rank, and a simple computation, we obtain that, 
\[
AT_0=Ax\otimes f-Ay\otimes g, \quad  \quad (AT_0)^2=Ax\otimes f+Ay\otimes g.
\]
Taking traces yields
\begin{equation}\label{eq:Tzero}
\tr(AT_0)=0,\qquad \tr((AT_0)^2)=2.
\end{equation}

\step{Step 2: The construction of \(T\).}

Since $A$ and $B$ are linearly independent, Lemma~\ref{lem:tracepairing}
implies that the linear map
\[
T\longmapsto \bigl(\ell_A(T),\ell_B(T)\bigr)
\]
from $\mathcal F$ onto $\mathbb C^2$ is surjective. Consequently,
there exists $T_1\in\mathcal F$ such that
\[
\ell_A(T_1)=0,\qquad \ell_B(T_1)=1.
\]
For $t\in\mathbb C$, define $T(t)=T_0+tT_1$ and set
$q(t):=\ell_B(T(t))$ and $p(t):=\operatorname{tr}((AT(t))^2)$.
By the choice of $T_0$ and $T_1$, we have, for every $t\in\mathbb C$,
$$
\ell_A(T(t))=0, \quad \text{ and} \quad q(t):=\ell_B(T(t))=\ell_B(T_0)+t,
$$
and using the linearity of the trace, we obtain
$$
p(t)
=\tr((AT_0)^2)
+t\tr(AT_0AT_1+AT_1AT_0)
+t^2\tr((AT_1)^2)
$$
Since $p(0)=2$, the polynomial $p$ is nonzero and has degree at
most two. As $q$ is also a nonzero polynomial, we can choose
$t_0\in\mathbb C$ such that $q(t_0)\ne0$ and $p(t_0)\ne0$.
Setting $T:=T(t_0)$, we obtain a finite-rank operator $T\in\F$
satisfying
\begin{equation}\label{eq:testthree}
\tr(AT)=0,\qquad \tr(BT)\ne0,
\qquad \tr((AT)^2)\ne0.
\end{equation}

Now, let $C=\alpha A+\beta B\ne0$ be an arbitrary  non zero element from $\mathcal{V}$. If $\beta\ne0$, then
\[
\tr(CT)=\alpha\tr(AT)+\beta\tr(BT)=\beta\tr(BT)\ne0.
\]
If $\beta=0$, then $\alpha\ne0$ and
\[
\tr((CT)^2)=\alpha^2\tr((AT)^2)\ne0.
\]
In either case, $CT$ is a finite-rank operator for which the trace of some positive integer power is nonzero. Hence $r(CT)>0$, as desired.
\end{proof}

\section{The algebraic reflexive hull of a subspace}

For two complex vectors spaces $U,V$ of  finite-dimensional  and for a linear subspace $\mathcal S\subseteq\mathcal L(U,V)$, we denote by $\AlgRef(\mathcal S)$ for the  algebraic reflexive hull of the subspace $\mathcal{S}$, that is,
\[
\AlgRef(\mathcal S)=\{C\in\mathcal L(U,V): Cx\in\mathcal Sx\text{ for every }x\in U\},
\qquad
\mathcal Sx=\{Ax:A\in\mathcal S\}.
\]
$\mathcal{S}$ is called to be algebraically reflexive when $\AlgRef(\mathcal S)=\mathcal S$. It is called reduced when
\[
\bigcap_{A\in\mathcal S}\ker A=\{0\},
\qquad
\sum_{A\in\mathcal S}\operatorname{ran}A=V.
\]
The notion of the algebraic reflexive hull was introduced by
Loginov and Shulman in \cite{loginov1975} and has since been used in the study of operator spaces by many authors; see, for instance,
\cite{DSP, BK, eleftheriakis2012, erdos1986, muller2009}.
We recall the following consequence over the complex field of the
classification of two-dimensional reduced subspaces of
$\mathcal \mathcal L (U,V)$, established in \cite[Theorem 6.1]{DSP};
see also \cite{BK}.
\begin{theorem}\cite{{DSP}}\label{known:reflexivity}
Let $U,V$ be complex finite-dimensional  vector spaces and let $\mathcal S\subseteq\mathcal L(U,V)$ be a reduced two-dimensional subspace. If $\AlgRef(\mathcal S)\ne\mathcal S$, then $\dim U=\dim V=2$ and, after choosing bases of $U$ and $V$, $\mathcal S$ is represented by
\[
\mathcal S=\spn{I_2,N},\qquad N=\begin{pmatrix}0&1\\0&0\end{pmatrix}.
\]
\end{theorem}

The following result plays a crucial role in the proofs of our intermediate lemmas. It provides a necessary and sufficient condition for three operators to be linearly dependent.

\begin{lemma}\label{lem:triple}
Let $A,B,C\in\B$. Suppose that
\begin{equation}\label{eq:local}
\spn{Ax,Bx}=\spn{Ax,Cx}=\spn{Bx,Cx}
\qquad(x\in X).
\end{equation}
Then $A,B,C$ are linearly dependent as operators.
\end{lemma}
\begin{proof}
Assume for a contradiction that $A,B,C$ are linearly independent.



For each \(x\in X\), define the linear map
\[
E_x:\mathbb C^3\longrightarrow X,\qquad
E_x(\alpha,\beta,\gamma)
=\alpha Ax+\beta Bx+\gamma Cx.
\]
Since \(A,B,C\) are linearly independent, the only coefficient vector \((\alpha,\beta,\gamma)\in\mathbb C^3\) satisfying
\[
\alpha Ax+\beta Bx+\gamma Cx=0
\qquad\text{for every }x\in X
\]
is the zero vector. Equivalently,
\[
\bigcap_{x\in X}\ker E_x=\{0\}.
\]

We claim that this intersection can be reduced to the intersection of finitely many of these kernels. Indeed, begin with \(K_0=\mathbb C^3\). Suppose that vectors \(x_1,\ldots,x_j\in X\) have been chosen and that
\[
K_j:=\bigcap_{i=1}^{j}\ker E_{x_i}\neq\{0\}.
\]
Choose a nonzero vector \((\alpha,\beta,\gamma)\in K_j\). By the linear independence of \(A,B,C\), implies that there exists \(x_{j+1}\in X\) such that
\(
\alpha Ax_{j+1}+\beta Bx_{j+1}
+\gamma Cx_{j+1}\neq0.
\)
Consequently, the vector 
\((\alpha,\beta,\gamma)\notin\ker E_{x_{j+1}}\), and therefore
\[
K_{j+1}:=K_j\cap\ker E_{x_{j+1}}
\]
is a proper subspace of \(K_j\). Hence \(\dim K_{j+1}\le \dim K_j-1\). Starting from \(\dim K_0=3\), we obtain \(K_3=\{0\}\), therefore there exist \(k\le 3\) and \(x_1,\dots,x_k\in X\) such that \(\bigcap_{i=1}^{k}\ker E_{x_i}=\{0\}\).

Now, Set
\[
U=\operatorname{span}\{x_1,\ldots,x_k\}.
\]
Then $U$ is a finite-dimensional subspace of $X$ with
$\dim U\leq 3$. Moreover, the restrictions
$A|_U$, $B|_U$, and $C|_U$ on  $U$ remain linearly independent.

Set
\[
V:=A(U)+B(U),\qquad
K:=\ker(A|_U)\cap\ker(B|_U).
\]
By \eqref{eq:local}, we have $C(U)\subseteq V$ and
$K\subseteq\ker(C|_U)$. Since $K$ is also contained in the
kernels of $A|_U$ and $B|_U$, the three operators induce
well-defined linear operators on the quotient space $U/K$  
\[
\widetilde A,\widetilde B,\widetilde C:U/K\longrightarrow V.
\]
Furthermore, these operators remain linearly independent.

Now, let  $\mathcal S=\spn{\widetilde A,\widetilde B}$. 
Let
\[
\mathcal S:=\operatorname{span}\{\widetilde A,\widetilde B\}.
\]
Since $\widetilde A$ and $\widetilde B$ are linearly independent,
$\dim\mathcal S=2$. By the definition of $K$, the common kernel
of the operators in $\mathcal S$ is trivial, while their ranges
span $V$. Hence, $\mathcal S$ is a reduced operator space.
Moreover, the preceding arguments yield
\[
\widetilde C\in\AlgRef(\mathcal S)
\setminus\mathcal S.
\]
We may therefore apply Theorem~\ref{known:reflexivity}, since
both $U/K$ and $V$ are finite-dimensional. It follows that,
with respect to suitable bases of these spaces, the operator
space $\mathcal S$ can be represented as
\[
\mathcal S=\operatorname{span}\{I_2,N\},
\qquad
N=\begin{pmatrix}0&1\\0&0\end{pmatrix}.
\]
For simplicity, use $A,B,C$ for the matrices of the induced maps $\widetilde A,\widetilde B$ and $\widetilde C$ in these bases. Since $A,B$ form a basis of $\mathcal S$, there exist $a,b,c,d\in\C$ such that
\[
A=aI_2+bN,\qquad B=cI_2+dN,
\qquad \Delta:=ad-bc\ne0.
\]
The inclusion $Ce_1\in\mathcal Se_1=\C e_1$ gives
\[
C=\begin{pmatrix}p&q\\0&r\end{pmatrix}
\]
for scalars $p,q,r$.  Otherwise, $ad-bc\neq 0$. Interchanging $A$ and $B$ if necessary, we may assume that $a\neq 0$. For every $s\in\mathbb C$, put $x=(s,1)$. Then \[
Ax=(as+b,a),\quad
Bx=(cs+d,c),\quad
Cx=(ps+q,r).
\]
In particular,
\[
\det(Ax,Bx)=bc-ad\neq 0.
\]
By condition (\ref{eq:local}), we obtain
\[
\det(Ax,Cx)
=a(r-p)s+br-aq\neq 0
\qquad(s\in\mathbb C).
\]
It follows that $a(r-p)=0$. Hence $r=p$,
and consequently
\(
C\in\operatorname{span}\{I_2,N\}
=\operatorname{span}\{A,B\},
\)
this contradict the  independence of the matrices $A, B$ and $C$, and this complete the required  proof.
\end{proof}

\section{Projective Space and Incidence condition}\label{sec:incidence}
For each nonzero operator $A\in\B$, we denote by
\[
[A]=\{tA:t\in\mathbb C\}
\]
the one-dimensional subspace spanned by $A$. Thus, $[A]=[B]$
if and only if $B=tA$ for some $t\in\mathbb C\setminus\{0\}$.
Given a linear subspace $\mathcal V\subseteq\B$, we denote by
$\PP(\mathcal V)$ its projective space, that is, the set of all
one-dimensional subspaces of $\mathcal V$.

For each nonzero $A\in\B$, define
\[
Z(A)=\{T\in\B:r(AT)=0\}.
\]
The following lemma shows that $Z(A)$ uniquely determines the
projective class of $A$.

\begin{lemma}\label{lem:profiles}
For any nonzero operators $A,B\in\B$, we have
\[
Z(A)=Z(B)\quad\Longleftrightarrow\quad [A]=[B].
\]
\end{lemma}

\begin{proof}
If $B=tA$ for some $t\ne0$, then
$r(BT)=|t|r(AT)$ for every $T\in\B$, and hence
$Z(A)=Z(B)$.

Conversely, suppose that $Z(A)=Z(B)$. Then, for every
$P\in\Ione$,
\[
r(AP)=0\quad\Longleftrightarrow\quad r(BP)=0.
\]
Since $AP$ and $BP$ have rank at most one, this is equivalent
to saying that $AP$ is nilpotent if and only if $BP$ is
nilpotent. By Lemma~\ref{known:determination}, $A$ and $B$
are linearly dependent. As both operators are nonzero, it
follows that $[A]=[B]$.
\end{proof}

We next observe that the condition $r(AB)=0$ depends only on
the projective classes of $A$ and $B$. Indeed, for any
$s,t\in\mathbb C\setminus\{0\}$,
\[
r\bigl((sA)(tB)\bigr)=|st|\,r(AB).
\]
Thus, the condition $r(AB)=0$ defines a relation
on $\PP(\B)$, leading to the following definition.

\begin{definition}\label{def:incidence}
Let $\mathcal L\subseteq\PP(\B)$ be a family of projective points. For $T\in\B$, define
\[
\mathcal L_T=\{[C]\in\mathcal L:r(CT)=0\}.
\]
We say that $\mathcal L$ satisfies the incidence condition if it has at least three points and, for every $T\in\B$, the set $\mathcal L_T$ is either a singleton or the whole family $\mathcal L$.
\end{definition}

The following theorem provides a necessary and sufficient condition for a family of projective points to satisfy the incidence condition.

\begin{theorem}\label{thm:incidence}
Let $\mathcal L\subseteq\PP(\B)$ contain at least three points. The following statements are equivalent:
\begin{enumerate}[label=\textup{(\roman*)}]
\item For every $T\in\B$, $\mathcal L_T$ is a singleton or $\mathcal L_T=\mathcal L$.
\item There is a two-dimensional linear space $\mathcal V\subseteq\B$ such that every nonzero member of $\mathcal V$ has rank one and $\mathcal L=\PP(\mathcal V)$.
\end{enumerate}
\end{theorem}
\begin{proof}
First, we show that $ (i)\Longrightarrow (ii)$.

Choose distinct points $[A],[B],[C]\in\mathcal L$ and nonzero representatives $A,B,C$. Fix $x\in X$ and let $f\in X^*$ satisfy $f(Ax)=f(Bx)=0$. For $T=x\otimes f$, formula~\eqref{eq:rankonetests} gives
\[
r(AT)=r(BT)=0.
\]
The set $\mathcal L_T$ contains the two distinct points $[A]$ and $[B]$, so it is not a singleton, thus $\mathcal L_T=\mathcal L$, and hence $r(CT)=0$ and $f(Cx)=0$.

We have shown that, for every $f\in X^*$, the conditions
$f(Ax)=f(Bx)=0$ imply $f(Cx)=0$. By the Hahn-Banach
theorem, it follows that  $Cx\in\spn{Ax,Bx}$. 
The same argument applies to the pairs $[A],[C]$ and $[B],[C]$ gives 
\[
Cx\in\spn{Ax,Bx},\quad
Bx\in\spn{Ax,Cx},\quad
Ax\in\spn{Bx,Cx}.
\]
These inclusions imply equality of the three pairwise spans. As $x$ was arbitrary, Lemma~\ref{lem:triple} shows that $A,B,C$ are linearly dependent. Hence, there exist scalars
$a,b,c\in\mathbb{C}$, not all zero, such that
\[
aA+bB+cC=0.
\]
Since $A$ and $B$ are linearly independent, we necessarily
have $c\neq0$. Set $\mathcal V:=\operatorname{span}\{A,B\},$ then 
$C=-\frac{a}{c} A-\frac{b}{c} B$, it follows that $C\in\mathcal V$. Applying the same argument to  every $[S]\in\mathcal L$,
we conclude that $S\in\mathcal V$ for every such $S$. Therefore, $\mathcal L\subseteq\PP(\mathcal V)$. 

We now show that every nonzero element of $\mathcal V$ has
rank one. Suppose, for a contradiction, that there exists
$D\in\mathcal V$ with $\operatorname{rank}(D)\geq 2$.
By Lemma~\ref{lem:plane}, there exists a finite-rank operator $T\in \F$ such that
\[
r(CT)>0,
\qquad\text{for every }C\in\mathcal V\setminus\{0\}.
\]
Since $\mathcal L\subseteq\mathbb P(\mathcal V)$, then   $\mathcal L_T=\varnothing$, and this contradicts condition~(i). Therefore, every nonzero element of $\mathcal V$ has rank one.

To complete the proof of the first implication, it remains to show that
$\mathcal L=\mathbb P(\mathcal V)$. Suppose, to the contrary,
that $\mathcal L$ is a proper subset of $\mathbb P(\mathcal V)$.
Choose $[D]\in\mathbb P(\mathcal V)\setminus\mathcal L$,
with $D\in\mathcal V\setminus\{0\}$, and let $E\in\mathcal V$
be linearly independent of $D$. By
Lemma~\ref{lem:tracepairing}, there exists $T\in\F$
such that
\[
\operatorname{tr}(DT)=0,\qquad
\operatorname{tr}(ET)=1.
\]
Since $D$ and $E$ form a basis of $\mathcal V$, every
$C\in\mathcal V$ can be written as $C=\alpha D+\beta E$.
Moreover, every nonzero element of $\mathcal V$ has rank one.
Therefore, $CT$ has rank at most one, and consequently
\[
r(CT)
=|\operatorname{tr}(CT)|
=|\alpha\operatorname{tr}(DT)
+\beta\operatorname{tr}(ET)|
=|\beta|.
\]
It follows that $r(CT)=0$ if and only if
$C$ and $ D$ are linearly dependent. Thus, $[D]$ is the unique point of
$\mathbb P(\mathcal V)$ annihilated by $T$. Since
$[D]\notin\mathcal L$, we obtain
$\mathcal L_T=\varnothing$, contradicting condition~(i).
Therefore,
\[
\mathcal L=\mathbb P(\mathcal V).
\]

Now, assume that (ii) holds. For each $T\in\B$, consider the
linear functional
\[
\lambda_T:\mathcal V\to\mathbb C,\qquad
\lambda_T(C)=\operatorname{tr}(CT).
\]
By \eqref{eq:rankone}, we have
\[
\mathcal L_T
=\{[C]\in\mathcal L:\lambda_T(C)=0\}.
\]
If $\lambda_T=0$, then $\mathcal L_T=\mathcal L$. Otherwise,
$\ker\lambda_T$ is one-dimensional, since $\dim\mathcal V=2$.
Hence, $\mathcal L_T$ contains at most one projective point. In either case, condition (i) follows.
\end{proof}
\begin{corollary}\label{cor:rankone}
Let $A\in\B\setminus\{0\}$. Then $\rank A=1$ if and only if $[A]$ belongs to a family $\mathcal L$ satisfying  the incidence condition.
\end{corollary}
\begin{proof}
If $[A]$ belongs to such a family, Theorem~\ref{thm:incidence} places $A$ in a two-dimensional space all of whose nonzero elements have rank one. Hence $\rank A=1$.

Conversely, write $A=x\otimes f$ with $x\ne0$ and $f\ne0$. Since $X$ is infinite-dimensional, its dual contains a functional $g$ independent of $f$. Let
\[
\mathcal V=\spn{x\otimes f,x\otimes g}.
\]
Clearly, $\mathcal V$ is a two-dimensional linear subspace. Moreover, every
nonzero element of $\mathcal V$ can be written as $x\otimes h$ with
$h\ne0$, and hence has rank one. By Theorem~\ref{thm:incidence},
$\PP(\mathcal V)$ satisfies the incidence condition and contains $[A]$.
\end{proof}

\section{The proof of the main Theorem}\label{sec:projective}
In this section, we prove our main theorem \ref{thm:main}. To this end, we first
establish several auxiliary lemmas, building on the results obtained
in the preceding sections, which will be used in the proof of
the theorem. For the remainder of this section, we assume that
$\phi:\B\to\B$ is a surjective map satisfying the condition \eqref{eq:P}.

\begin{lemma}\label{lem:zero}
We have $\phi(0)=0$, and $\phi(A)=0$ implies $A=0$. Moreover,
\[
A\in\Q\quad\Longleftrightarrow\quad\phi(A)\in\Q.
\]
\end{lemma}
\begin{proof}
First observe that an operator $D$ satisfying $r(DT)=0$ for every $T\in\B$ must be zero. Indeed, by \eqref{eq:rankonetests}, testing $T=x\otimes f$ gives $f(Dx)=0$ for all $x\in X$ and $f\in X^*$. Hahn--Banach gives $Dx=0$ for every $x$.

Applying the preceding argument to $D=\phi(A)$ and using the
surjectivity of $\phi$, we obtain
\begin{align*}
\phi(A)=0
&\iff r\bigl(\phi(A)\phi(B)\bigr)=0
\quad\text{for every }B\in\B\\
&\iff r(AB)=0
\quad\text{for every }B\in\B\\
&\iff A=0.
\end{align*}
Now, apply \eqref{eq:P} with the pair $(A,A)$, it follows that
\[
r(A)=0\iff r(A^2)=0\iff r(\phi(A)^2)=0\iff r(\phi(A))=0.
\]
\end{proof}
\begin{lemma}\label{prop:Phi}
The map
\[
\Phi:\PP(\B)\to\PP(\B),\qquad \Phi([A])=[\phi(A)],
\]
is a well-defined bijection that preserves the relation
$r(AB)=0$ in both directions.
\end{lemma}

\begin{proof}
By Lemma~\ref{lem:zero}, $\phi(A)\ne0$ whenever $A\ne0$. Thus, $\Phi([A])=[\phi(A)]$ is well-defined as a projective class.

To show that $\Phi$ is well-defined, suppose that $[A]=[B]$.
Then $Z(A)=Z(B)$ by Lemma~\ref{lem:profiles}. Given any
$U\in\B$, surjectivity of $\phi$ yields $T\in\B$ such that
$U=\phi(T)$. Hence, by \eqref{eq:P},
\[
r(\phi(A)U)=0
\iff r(AT)=0
\iff r(BT)=0
\iff r(\phi(B)U)=0.
\]
Thus, $Z(\phi(A))=Z(\phi(B))$, and Lemma~\ref{lem:profiles}
implies that $[\phi(A)]=[\phi(B)]$. Therefore, $\Phi$ is
well-defined.

To prove that $\Phi$ is bijective, we first show that it is injective. Suppose that $[\phi(A)]=[\phi(B)]$ for some
$A,B\in\B$.  Then
$Z(\phi(A))=Z(\phi(B))$ by Lemma~\ref{lem:profiles}.
Using \eqref{eq:P}, we obtain, for every $T\in\B$,
\[
r(AT)=0\iff r(\phi(A)\phi(T))=0
\iff r(\phi(B)\phi(T))=0\iff r(BT)=0.
\]
Hence $Z(A)=Z(B)$ and $[A]=[B]$. Therefore $\Phi$ is injective. 

Finally, the surjectivity of $\phi$  implies directly  that of $\Phi$. Consequently, $\Phi$ is a bijective maps.
\end{proof}

\begin{lemma}\label{prop:rankone}
For every nonzero $A\in\B$,
\[
\rank A=1\quad\Longleftrightarrow\quad\rank\phi(A)=1.
\]
Consequently, $\phi$ maps $\Rone$ onto $\Rone$ and maps the rank-one nilpotents onto the rank-one nilpotents.
\end{lemma}
\begin{proof}
For any family $\mathcal L\subseteq\PP(\B)$ and any $T\in\B$,
we have
\begin{equation}\label{eq:transport}
\Phi(\mathcal L_T)=\bigl(\Phi(\mathcal L)\bigr)_{\phi(T)}.
\end{equation}
Indeed, since $\Phi$ is bijective, every point of $\Phi(\mathcal L)$
has a unique representation $\Phi([C])$ with $[C]\in\mathcal L$.
Moreover, by \eqref{eq:P},
\[
\begin{aligned}
\Phi([C])\in\Phi(\mathcal L_T)
&\iff [C]\in\mathcal L_T\\
&\iff r(CT)=0\\
&\iff r\bigl(\phi(C)\phi(T)\bigr)=0\\
&\iff \Phi([C])\in\bigl(\Phi(\mathcal L)\bigr)_{\phi(T)}.
\end{aligned}
\]
Therefore, \eqref{eq:transport} follows.

Suppose that $\mathcal L$ satisfies the incidence condition.
Since $\Phi$ is injective, $\Phi(\mathcal L)$ also contains at least
three points. Given any $U\in\B$, choose $T\in\B$ such that
$\phi(T)=U$. By \eqref{eq:transport},
\[
\Phi(\mathcal L_T)=\bigl(\Phi(\mathcal L)\bigr)_U.
\]
Since $\Phi$ is bijective, it preserves singletons and whole
families. Hence, $\mathcal L_T$ is a singleton or equals
$\mathcal L$ if and only if $\bigl(\Phi(\mathcal L)\bigr)_U$
is a singleton or equals $\Phi(\mathcal L)$, respectively.
Thus, $\Phi(\mathcal L)$ satisfies the incidence condition.  
Conversely, if a family $\mathcal M$ satisfies the condition, put $\mathcal L=\Phi^{-1}(\mathcal M)$. For each $T$, equation~\eqref{eq:transport} gives $\Phi(\mathcal L_T)=\mathcal M_{\phi(T)}$, which is a singleton or $\mathcal M$.  Therefore the condition is also preserved by $\Phi^{-1}$.


Applying Corollary~\ref{cor:rankone} in both directions shows
that $\phi$ preserves rank-one operators in both directions.
Consequently, its restriction to the set of rank-one operators
is surjective. Moreover, since a rank-one operator is
quasi-nilpotent if and only if it is nilpotent, Lemma~\ref{lem:zero}
implies that $\phi$ also preserves rank-one nilpotents in both
directions. By surjectivity, the induced map on rank-one
nilpotents is onto.
\end{proof}

\begin{lemma}\label{prop:adjacency}
For all $R_1,R_2\in\Rone$,
\[
R_1\sim R_2\quad\Longleftrightarrow\quad
\phi(R_1)\sim\phi(R_2).
\]
\end{lemma}
\begin{proof}
If $R_1$ and $R_2$ are linearly dependent, the result follows directly from Lemma~\ref{prop:Phi}. The same proposition implies that if $\phi(R_1)$ and $\phi(R_2)$ are linearly dependent, then the corresponding rank-one operators are as well. We may therefore assume that $R_1$ and $R_2$ are linearly independent.

Suppose first that $R_1\sim R_2$. By
Lemma~\ref{known:adjacency}, there exists a rank-one operator
$R_3$, non proportional to both $R_1$ and $R_2$, such that
\[
r(R_1T)=r(R_2T)=0
\quad\Longrightarrow\quad r(R_3T)=0
\qquad(T\in\B).
\]
Let $U\in\B$ satisfy
$r(\phi(R_1)U)=r(\phi(R_2)U)=0$. By surjectivity of $\phi$,
write $U=\phi(T)$ for some $T\in\B$. Using \eqref{eq:P}, we obtain
$r(R_1T)=r(R_2T)=0$, and hence $r(R_3T)=0$. Another application
of \eqref{eq:P} gives $r(\phi(R_3)U)=0$. Since, $\phi(R_3)$ is also of rank-one, then Lemma~\ref{known:adjacency} implies that
$$\phi(R_1)\sim\phi(R_2).$$

Conversely, suppose that $\phi(R_1)\sim\phi(R_2)$. By
Lemma~\ref{known:adjacency}, there exists a rank-one operator $D$,
nonproportional to either $\phi(R_1)$ or $\phi(R_2)$, such that
\[
r(\phi(R_1)U)=r(\phi(R_2)U)=0
\quad\Longrightarrow\quad r(DU)=0
\qquad(U\in\B).
\]
By surjectivity of $\phi$, choose $R_3\in\B$ such that
$\phi(R_3)=D$. Proposition~\ref{prop:rankone} implies that
$R_3$ has rank one, and injectivity of $\Phi$ ensures that
$R_3$ is non proportional to $R_1$ and $R_2$. For any $T\in\B$,
if $r(R_1T)=r(R_2T)=0$, then \eqref{eq:P} yields
$r(\phi(R_1)\phi(T))=r(\phi(R_2)\phi(T))=0$. The witness property
of $D$ gives $r(D\phi(T))=0$, and hence $r(R_3T)=0$. Lemma~\ref{known:adjacency}  implies
$R_1\sim R_2$, and this  complete the proof. 
\end{proof}

\begin{lemma}\label{prop:xi}
For every $P \in \mathcal{I}_1$, let us define
\begin{equation}\label{eq:xi}
\xi(P)=\frac{\phi(P)}{\tr(\phi(P))}.
\end{equation}
Then, the maps   $\xi:\Ione\to\Ione$ establishes a bijection of $\Ione$  with itself.
\end{lemma}

\begin{proof}
Let $P\in\Ione$. Since $\rank P=1$ and $r(P)=1$,
Lemmas~\ref{lem:zero} and Lemma~\ref{prop:rankone}
imply that $\phi(P)$ has rank one and is not quasinilpotent.
Thus, $\tr(\phi(P))\ne0$ by \eqref{eq:rankone}. Setting
$c=\tr(\phi(P))$, we obtain
\[
\xi(P)^2
=\frac{\phi(P)^2}{c^2}
=\frac{c\phi(P)}{c^2}
=\xi(P).
\]
Since $\xi(P)$ has rank one, it belongs to $\Ione$.

To prove injectivity, suppose that $\xi(P)=\xi(Q)$. Then
$[\phi(P)]=[\phi(Q)]$, and Lemma~\ref{prop:Phi} yields
$[P]=[Q]$. Hence, $Q=tP$ for some $t\ne0$. Since $P$ and $Q$
are idempotents, $t^2P=tP$, so $t=1$ and $P=Q$.

Finally, let $Q\in\Ione$. By surjectivity of $\phi$, choose
$A\in\B$ such that $\phi(A)=Q$. By
Lemma~\ref{prop:rankone} and Lemma~\ref{lem:zero},
$A$ has rank one and is not quasinilpotent. Thus,
$a=\tr(A)\ne0$, and $P=A/a$ belongs to $\Ione$. Since
$[P]=[A]$, then 
$[\phi(P)]=[\phi(A)]=[Q]$. Consequently, $\xi(P)$ and $Q$
are idempotents in the same projective class, which implies
$\xi(P)=Q$ by the injectivity argument above. Therefore,
$\xi$ is surjective, and this  complete the proof.
\end{proof}


\begin{lemma}\label{prop:orthogonality}
The bijection $\xi$ in \eqref{eq:xi} preserves orthogonality
in both directions.
\end{lemma}

\begin{proof}
Let $P,Q\in\Ione$ be distinct. Suppose first that
$PQ=QP=0$. By Lemma~\ref{known:orthogonality}, there exist
rank-one nilpotents $M,N$ such that
\[
P\sim M,\quad P\sim N,\quad Q\sim M,\quad Q\sim N,
\qquad M\not\sim N.
\]
By Lemmas~\ref{prop:rankone} and~\ref{prop:adjacency}, these
relations are preserved by $\phi$. Since multiplying an operator
by a nonzero scalar does not change its range or kernel, the same
relations hold with $\xi(P)$ and $\xi(Q)$ in place of $P$ and $Q$.
Lemma~\ref{known:orthogonality} therefore yields
\[
\xi(P)\xi(Q)=\xi(Q)\xi(P)=0.
\]

Conversely, suppose that $\xi(P)\xi(Q)=\xi(Q)\xi(P)=0$.
Applying Lemma~\ref{known:orthogonality} to $\xi(P)$ and $\xi(Q)$,
we obtain rank-one nilpotents $D,E$ satisfying the corresponding
four adjacency relations and $D\not\sim E$. By surjectivity of
$\phi$, choose $M,N\in\B$ such that $\phi(M)=D$ and $\phi(N)=E$.
By Lemmas~\ref{prop:rankone} and~\ref{prop:adjacency}, the operators
$M,N$ are rank-one nilpotents and satisfy the corresponding
relations with $P,Q$, with $M\not\sim N$. Hence,
Lemma~\ref{known:orthogonality} gives
\[
PQ=QP=0.
\]
 Thus, $\xi$ preserves orthogonality in both directions.
\end{proof}

We are now ready to prove the main theorem. It suffices to establish the direct implication. 

\begin{proof}[Proof of Theorem \ref{thm:main}]
By Theorem~\ref{known:semrl}, the bijection $\xi$ defined in
\eqref{eq:xi} takes one of the following forms:
\[
\xi(P)=SPS^{-1}
\quad\text{or}\quad
\xi(P)=SP'S^{-1},
\qquad P\in\Ione.
\]
Without loss of generality, assume that $\xi(P)=SPS^{-1}$.
Then
\[
r\bigl((SAS^{-1})\xi(P)\bigr)=r(AP).
\]
Fix $A\ne0$. By the definition of $\xi$,
\[
\phi(P)=\tr(\phi(P))\,\xi(P),
\qquad P\in\Ione,
\]
where $\tr(\phi(P))\ne0$. Hence, by \eqref{eq:P},
\[
r\bigl(\phi(A)\xi(P)\bigr)=0
\iff r(AP)=0
\iff r\bigl((SAS^{-1})\xi(P)\bigr)=0.
\]
Since $\xi$ maps $\Ione$ onto itself, we obtain
\[
r(\phi(A)Q)=0
\iff r(SAS^{-1}Q)=0,
\qquad Q\in\Ione.
\]
By Lemma~\ref{known:determination}, the nonzero operators
$\phi(A)$ and $SAS^{-1}$ are linearly dependent. Thus there exists
a function $\gamma:\B\setminus\{0\}\to\C^\times$ such that
\[
\phi(A)=\gamma(A)SAS^{-1},
\qquad A\ne0.
\]
The second case, where $\xi(P)=SP'S^{-1}$, follows similarly. 

To complete the proof, it remains to show that
\[
\{\gamma(tA)\sigma(t):t\in\C^\times\}=\C^\times,
\qquad A\in\B\setminus\{0\}.
\]
To this end, write
\[
\phi(A)=\gamma(A)\psi(A),\qquad A\ne0,
\]
where
\[
\psi(A)=SAS^{-1}
\quad\text{or}\quad
\psi(A)=SA^*S^{-1}.
\]
Fix $A\ne0$ and let $s\in\C^\times$. By the surjectivity of
$\phi$, there exists $B\ne0$ such that
\(
\phi(B)=s\psi(A).
\) 
Since $\phi(B)=\gamma(B)\psi(B)$, then  we obtain $\gamma(B)\psi(B)=s\psi(A).$ Therefore, $\psi(B)$ and $\psi(A)$ are linearly dependent, which implies that $A$ and $B$ are linearly dependent as well. Hence, there exists some $t \in \mathbb{C}^\times$ such that $B = tA$.
Consequently,
\[
s\psi(A)=\phi(tA)
=\gamma(tA)\psi(tA)
=\gamma(tA)\sigma(t)\psi(A).
\]
Since $\psi(A)\ne0$, we conclude that
\(
s=\gamma(tA)\sigma(t).
\)
As $s\in\C^\times$ was arbitrary, this proves that
\[
\{\gamma(tA)\sigma(t):t\in\C^\times\}=\C^\times,
\qquad A\in\B\setminus\{0\},
\]
and completes the proof of our main Theorem \ref{thm:main}.
\end{proof}

\begin{remark}~~
\begin{enumerate}
    \item In general, a surjective map $\phi:\B\to\B$ satisfying
    condition~\eqref{eq:P} need not be linear or continuous. For
    instance, the map defined by
    \[
    \phi(0)=0,\qquad \phi(A)=\|A\|A\quad(A\ne0),
    \]
    satisfies~\eqref{eq:P}, but is not additive.
    Thus, additivity and continuity  cannot be deduced from the condition and surjectivity.

    \item In infinite-dimensional setting, a surjective map
    $\phi:\B\to\B$ preserves the nilpotency of product of  operator, 
    if and only if,  it preserves the quasi-nilpotency of product of  operator. Moreover, these conditions hold if and only if
    $\phi$ has one of the forms described in
    Theorem~\ref{thm:main}.
\end{enumerate}
\end{remark}


\begin{thebibliography}{9}
\bibitem{AAM}
Z. E. A. Abdelali, A. Achchi and R. Marzouki,
\emph{Maps preserving the local spectral radius zero of generalized product of operators},
Linear Multilinear Algebra \textbf{67} (2019), 2021--2029.
\href{https://doi.org/10.1080/03081087.2018.1479371}{doi:10.1080/03081087.2018.1479371}.

\bibitem{BK}
J. Bra\v ci\v c and B. Kuzma,
\emph{Reflexivity defect of spaces of linear operators},
Linear Algebra Appl. \textbf{430} (2009), 344--359.
\href{https://doi.org/10.1016/j.laa.2008.07.024}{doi:10.1016/j.laa.2008.07.024}.

\bibitem{DSP}
C. de Seguins Pazzis,
\emph{Range-compatible homomorphisms over the field with two elements},
Electron. J. Linear Algebra \textbf{34} (2018), 71--114.
\href{https://doi.org/10.13001/1081-3810.2982}{doi:10.13001/1081-3810.2982}.

\bibitem{eleftheriakis2012}
G.~K. Elefther\'akis and I.~G. Todorov,
\textit{Ranges of bimodule projections and reflexivity},
J. Funct. Anal. \textbf{262} (2012), no.~11, 4891--4915.

\bibitem{erdos1986}
J.~A. Erd\H{o}s,
\textit{Reflexivity for subspace maps and linear spaces of operators},
Proc. London Math. Soc. (3) \textbf{52} (1986), 582--600.
\href{https://doi.org/10.1112/plms/s3-52.3.582}{doi:10.1112/plms/s3-52.3.582}.

\bibitem{LSS}
C.-K. Li, P. \v Semrl and N.-S. Sze,
\emph{Maps preserving the nilpotency of products of operators},
Linear Algebra Appl. \textbf{424} (2007), 222--239.
\href{https://doi.org/10.1016/j.laa.2006.11.013}{doi:10.1016/j.laa.2006.11.013}.

\bibitem{loginov1975}
A.~I. Loginov and V.~S. Shulman,
\emph{Hereditary and intermediate reflexivity of $W^*$-algebras},
Math. USSR-Izv. \textbf{9} (1975), no.~6, 1189--1201.
\href{https://doi.org/10.1070/IM1975v009n06ABEH001517}{doi:10.1070/IM1975v009n06ABEH001517}.

\bibitem{muller2009}
V.~M\"uller,
\textit{On the reflexivity deficit},
in: Banach Algebras 2009 (abstract), Stefan Banach International Mathematical Center,
B\k{e}dlewo, 2009, pp.~14--24.

\bibitem{Semrl}
P. \v Semrl,
\emph{Non-linear commutativity preserving maps},
Acta Sci. Math. (Szeged) \textbf{71} (2005), 781--819.
\end{thebibliography}
\end{document}